\documentclass[a4paper]{amsart}
\usepackage{amssymb, amsmath, amscd, amsthm}
\usepackage{lineno}
\usepackage{amsfonts}
\usepackage{psfrag}
\usepackage[all]{xy}
\usepackage{graphicx}

\newtheorem{thm}{Theorem}[section]

\newtheorem{lem}[thm]{Lemma}

\theoremstyle{definition}

\newtheorem{defi}[thm]{Definition}
\newtheorem{exa}{\bf Example}
\theoremstyle{remark}
\newtheorem{rem}[thm]{Remark}

\newcommand{\eps}{\varepsilon}
\newcommand{\p}{\varphi}
\newcommand{\R}{\mathbb{R}}

\newcommand{\N}{\mathbb{N}}

\newcommand{\K}{\mathcal{K}}
\newcommand{\C}{\mathcal{C}}

\newcommand{\vp}{\vec \varphi}

\newcommand\bigfrown[2][\textstyle]{\ensuremath{%
  \!\!\!\array[b]{c}\text{\scalebox{2}{$\frown$}}\\[-1.3ex]#1#2\endarray\!\!\!}}

\begin{document}


\title[Filtrations induced by continuous functions]{Filtrations induced by continuous functions}

\author[B. Di Fabio]{B. Di Fabio}
\author[P. Frosini]{P. Frosini}
\address{ARCES and Dipartimento di Matematica, Universit\`a di Bologna, Italia}
\email{\{barbara.difabio,patrizio.frosini\}@unibo.it}



\begin{abstract}
In Persistent Homology and Topology, filtrations are usually given
by introducing an ordered collection of sets or a continuous
function from a topological space to $\R^n$. A natural question
arises, whether these approaches are equivalent or not. In this
paper we study this problem and prove that, while the answer to
the previous question is negative in the general case, the
approach by continuous functions is not restrictive with respect
to the other, provided that some natural stability and
completeness assumptions are made. In particular,  we show that
every compact and stable $1$-dimensional filtration of a compact
metric space is induced by a continuous function. Moreover, we
extend the previous result to the case of multi-dimensional
filtrations, requiring that our filtration is also complete. Three
examples show that we cannot drop the assumptions about stability
and completeness. Consequences of our results on the definition of
a distance between filtrations are finally discussed.
\end{abstract}

\maketitle

\section*{Introduction}
The concept of filtration is the start point for Persistent
Topology and Homology. Actually, the main goal of these theories
is to examine the topological and homological changes that happen
when we go through a family of spaces that is totally ordered with
respect to inclusion \cite{EdHa09}. In literature, filtrations are
usually given in two ways. The former consists of explicitly
introducing a nested collection of sets (usually carriers of
simplicial complexes), the latter of giving a continuous function
from a topological space to $\R$ or $\R^n$ (called a
\emph{filtering function}), whose sub-level sets represent the
elements of the considered filtration (cf., e.g.,
\cite{EdHa08,Gh08}). An example of these two types of filtrations
is shown in Figure \ref{Esempio}. The two considered methods have
produced two different approaches to study the concept of
persistence. A natural question arises, whether these approaches
are equivalent or not. In our paper we study this problem and
prove that, while the answer to the previous question is negative
in the general case, the approach by continuous functions is not
restrictive with respect to the other, provided that some natural
stability and completeness assumptions are made. In some sense,
this statement shows that the approach by continuous functions
(and the related theoretical properties) can be used without loss
of generality, and represents the main result of this paper.

The interest in this investigation is mainly due to the desire of
building a bridge between the two settings, which would ensure
that results available in literature for the approach by functions
are also valid for the other method. As examples of results that
have been proved in one setting and that it would be desirable to
apply to the other, we can cite \cite{CoEdHa07} and
\cite{CeDi*10}, in which persistence diagrams in the 1-dimensional
and $n$-dimensional setting, respectively, are proved to be stable
shape descriptors via the use of the associated filtering
functions. Another example can be found in \cite{DiLa10}, where a
Mayer-Vietoris formula involving the ranks of persistent homology
groups of a space and its subspaces is obtained by defining a
filtering function for the union space and taking account of its
restrictions to the considered subspaces.
\begin{figure}[htbp]
\psfrag{a}{$\subseteq$}\psfrag{h1}{$\C_1$}\psfrag{h2}{$\C_2$}\psfrag{h3}{$\C_3$}\psfrag{h4}{$\C_4$}
\psfrag{k1}{$\K_1$}\psfrag{k2}{$\K_2$}\psfrag{k3}{$\K_3$}\psfrag{k4}{$\K_4$}\psfrag{f}{$\p$}
\begin{center}
\includegraphics[width=\textwidth]{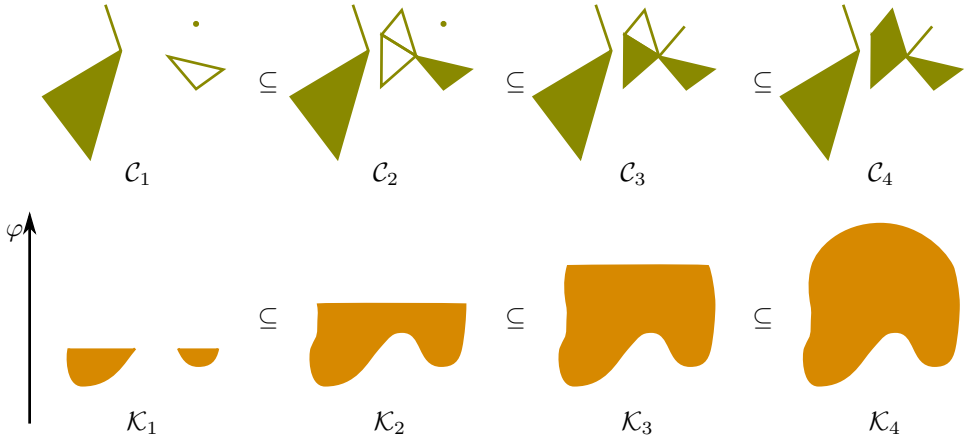}
\end{center}
\caption{\footnotesize{Examples of filtrations. First row: nested
carriers of simplicial complexes $\{\C_i\}$. Second row: the
sub-level sets $\{\K_i\}$ of a real-valued continuous function
$\p$.}}\label{Esempio}
\end{figure}

Another important reason which drives our investigation is related
to the problem of defining a distance between different
filtrations of the same space. Nowadays, this problem is usually
tackled by translating the direct comparison between two
filtrations into the comparison of the associated persistence
diagrams through the study of persistent homology. Unfortunately,
there exist some simple examples showing that this kind of
comparison is not always able to distinguish two different
filtrations (see e.g. \cite{CaFePo01,FrLa11,DiLa12}). For this
reason, our idea is to define a distance between filtrations in
terms of a distance between the associated filtering functions,
and to this scope, we need to prove that each filtration is
induced by at least one function (see Section \ref{distance}).

In this paper we just consider stable filtrations. The property of
stability of a filtration we ask for is motivated by the fact that
in real applications we need to work with methodologies that are
robust in the presence of noise. As a consequence, we have to
require that the inclusions considered in our filtration persist
under the action of small perturbations. For the same reason, we
also need that a small change of the parameter in our filtration
(whenever applicable) does not produce a large change of the
associated set with respect to the Hausdorff distance. These
assumptions are formalized by our definition of stable filtration
(Definition \ref{stableDef}).

In order to make our treatment as general as possible, we just
require that the sets $\K_i$ ($i\in I$) in our filtration are
compact subsets of a compact metric space $K$, and that the
indexing set $I$ is compact.

The paper starts by considering filtrations indexed by a
1-dimensional parameter. In this setting, after proving some
lemmas, we show that every compact and stable 1-dimensional
filtration of a compact metric space is induced by a continuous
function (Theorem \ref{monofil}). In the last part of the paper,
this result is extended to the case of multi-dimensional
filtrations (Theorem \ref{multfil}), i.e. the case of filtrations
indexed by an $n$-dimensional parameter (cf. \cite{Ca09,CaZo07}).
In order to do that, we need to assume also that our filtration is
complete, i.e. compatible with respect to intersection (Definition
\ref{defComplete}). Three examples show that we cannot drop either
the assumption about stability or the one concerning completeness
(Examples \ref{exastab-a}, \ref{exastab-b} and \ref{exacomplete}).
Some considerations on the consequences of our results conclude
the paper.

\section{Preliminaries}

In this section we give the preliminary concepts and the notation that
will be used throughout the paper.


Let $(K,d)$ be a non-empty compact metric space. Let us denote by
$Comp(K)$ the set $\{\K: \K\,\mbox{compact in}\,K\}$. Let us
consider the Hausdorff distance $d_H$ on $Comp(K)\setminus
\{\emptyset\}$. Moreover, let $I$ be a non-empty subset of $\R^n$
such that $I=I_1\times I_2\times\ldots\times I_n$. The following
relation $\preceq$ is defined in $I$: for $i=(i_1,\dots,i_n),
i'=(i'_1,\dots,i'_n)\in I$, we say $i\preceq i'$ if and only if
$i_r\leq\ i'_r$ for every $r=1,\dots,n$.

\begin{defi}
\label{deffilt} An $n$-dimensional \emph{filtration} of $K$ is an
indexed family $\{\K_{i}\in Comp(K)\}_{i\in I}$ such that,
$\emptyset,K\in\{\K_i\}_{i\in I}$, and $\K_{i}\subseteq \K_{i'}$
for every $i, i'\in I$, with $i\preceq i'$.
\end{defi}

\begin{defi}
An $n$-dimensional filtration $\{\K_{i}\}_{i\in I}$ of $K$ is
\emph{induced by a function} $\vp:K\to\R^n$ if $\K_{i}=\{P\in K:
\vp(P)\preceq i\}$ for every $i\in I$.
\end{defi}

\begin{defi}\label{compact}
We shall call \emph{compact}, or \emph{finite} any filtration
$\{\K_{i}\}_{i\in I}$ with $I=I_1\times I_2\times\ldots\times I_n$
a compact, or finite subset of $\R^n$, respectively.
\end{defi}

\begin{rem}
When $I$ is bounded, the assumption that
$\emptyset,K\in\{\K_i\}_{i\in I}$ is not so restrictive, since
each family of compact sets verifying the last property in
Definition~\ref{deffilt} can be extended to a family containing
$\emptyset$ and $K$, without losing that property. This assumption
allows us a more concise exposition.
\end{rem}

\section{Mono-dimensional filtrations}\label{Mono}
This section is devoted to prove our main result in the case of
filtrations indexed by a 1-dimensional parameter (Theorem
\ref{monofil}). Therefore, in what follows, the symbol $I$ will denote a
non-empty subset of $\R$.

For every subset $X\subseteq K$, let us denote by $\overline{X}$,
$int(X)$, $\partial X$, and $X^{c}$ the closure, the interior, the
boundary, and the complement of $X$ in $K$, respectively. We
recall that $int(X)^c=\overline{X^c}$.

\begin{defi}\label{stableDef}
We shall say that a compact $1$-dimensional filtration
$\{\K_i\}_{i\in I}$ of $K$ is \emph{stable with respect to the
metric} $d$ if the following statements hold:
\begin{itemize}
\item[$(a)$] The functions $i\mapsto\K_i$ and $i\mapsto\overline{\K_{i}^c}$ are continuous, i.e. if $(i_m\in I)_{m\in \N}$ is a sequence converging to
$\bar \imath\in I$, the sequence $(\K_{i_m})$ converges to
$\K_{\bar \imath}$, and the sequence $(\overline{\K_{i_m}^c})$
converges to $(\overline{\K_{\bar \imath}^c})$ with respect to the
Hausdorff distance $d_H$. \item[$(b)$] For every set
$\K_i$ and every $j\in I$ with $i<j$, $\K_i\subseteq int(\K_j)$.
\end{itemize}
\end{defi}

\begin{rem}
Let us note that in the case $\{\K_i\}_{i\in I}$ is a finite
$1$-dimensional filtration of $K$, Definition \ref{stableDef}
reduces to Definition \ref{stableDef}~$(b)$.
\end{rem}

\begin{rem}
We observe that, in Definition \ref{stableDef} $(a)$, the
convergence of the sequence $\left(\K_{i_m}\right)$ does not imply
the convergence of the sequence
$\left(\overline{\K_{i_m}^c}\right)$. Indeed, for example, let us
consider the following compact filtration of the set $K=[0,1]\cup
\{2\}$ ($K$ is endowed with the Euclidean metric). We take
$I=\{-1\}\cup [0,1]\cup \{2\}$ and set $\K_i=[0,i]$ for $i\in
[0,1]$, while $\K_{-1}=\emptyset$ and $\K_{2}=K$. It is immediate
to check that the sequence $\left(\K_{1-1/m}\right)$ converges to
$\K_{1}$, but the sequence $\left(\overline{\K_{1-1/m}^c}\right)$
does not converge to $\overline{\K_{1}^c}$.
\end{rem}

The following Lemmas \ref{ABcompact}--\ref{alphabeta} provide
meaningful properties of two functions $\alpha,\beta:K\to I$ which
turn out to be useful in the proof of our main result.

\begin{lem}\label{ABcompact}
Let $\{\K_i\}_{i\in I}$ be a compact and stable 1-dimensional
filtration of $K$. For every $P\in K$, let $A(P)=\{i\in I, P\in
\overline{\K_i^c}\}=\{i\in I, P\notin int(\K_i)\}$ and
$B(P)=\{i\in I, P\in\K_i\}$. Then $A(P)$ and $B(P)$ are non-empty
subsets of $I$. Moreover, $\sup A(P)\in A(P)$ and $\inf B(P)\in
B(P)$.
\end{lem}
\begin{proof}
First of all let us observe that both $A(P)$ and $B(P)$ are
non-empty because $i_{\min}=\min\{i\in I\}\in A(P)$ since
$\emptyset=\K_{i_{\min}}\in\{\K_i\}_{i\in I}$, and
$i_{\max}=\max\{i\in I\}\in B(P)$ since
$K=\K_{i_{\max}}\in\{\K_i\}_{i\in I}$.

Let $\alpha(P)=\sup A(P)$. Because of the compactness of $I$,
$\alpha(P)\in I$ and is finite. Let us show that $\alpha(P)\in
A(P)$. Let $(i_r)$ be a non-decreasing sequence of indices of
$A(P)$ converging to $\alpha(P)$. From
Definition~\ref{stableDef}~$(a)$, it follows that
$(\overline{\K_{i_r}^c})$ converges to
$\overline{\K_{\alpha(P)}^c}$. We have to prove that $\alpha(P)\in
A(P)$, i.e. $P\in\overline{\K_{\alpha(P)}^c}$. By contradiction,
let us assume that $P\notin\overline{\K_{\alpha(P)}^c}$. Since
$\overline{\K_{\alpha(P)}^c}$ is compact,
$d(P,\overline{\K_{\alpha(P)}^c})>0$. Therefore, for any
large enough index $r$, the inequality
$d_H(\overline{\K_{\alpha(P)}^c},\overline{\K_{i_r}^c})<d(P,\overline{\K_{\alpha(P)}^c})$
holds. Hence
$d(P,\overline{\K_{i_r}^c})\ge d(P,\overline{\K_{\alpha(P)}^c})-d_H(\overline{\K_{\alpha(P)}^c},\overline{\K_{i_r}^c})>0$ for any large enough index $r$, contrarily to our
assumption that $i_r\in A(P)$, i.e. $P\in\overline{\K_{i_r}^c}$.

Let $\beta(P)=\inf B(P)$. Because of the compactness of $I$,
$\beta(P)\in I$ and is finite. Let us show that $\beta(P)\in
B(P)$. Let $(i_r)$ be a non-increasing sequence of indices of
$B(P)$ converging to $\beta(P)$. From
Definition~\ref{stableDef}~$(a)$, it follows that $(\K_{i_r})$
converges to $\K_{\beta(P)}$. We have to prove that $\beta(P)\in
B(P)$, i.e. $P\in\K_{\beta(P)}$. By contradiction, let us assume
that $P\notin\K_{\beta(P)}$. Since $\K_{\beta(P)}$ is compact,
$d(P,\K_{\beta(P)})>0$. Therefore, for any large enough
index $r$, the inequality
$d_H\left(\K_{\beta(P)},\K_{i_r}\right)<d\left(P,\K_{\beta(P)}\right)$
holds. Hence
$d(P,\K_{i_r})>0$ for any large enough index $r$, contrarily to our
assumption that $i_r\in B(P)$, i.e. $P\in \K_{i_r}$.
\end{proof}

By virtue of the above Lemma \ref{ABcompact}, for every $P\in K$,
we can define $\alpha(P)=\max A(P)\in I$ and $\beta(P)=\min
B(P)\in I$. In plain words, for every $P\in K$, $\K_{\alpha(P)}$
is the largest compact $\K_i$ in the filtration such that
$P\in\overline{\K_i^c}=
int\left(\K_i\right)^c$, while $\K_{\beta(P)}$ is the
smallest compact $\K_j$ in the filtration such that $P\in \K_j$.
In particular, $P\in\overline{\K_{\alpha(P)}^c}\cap \K_{\beta(P)}$.

\begin{lem}\label{lemmaalphabeta}
Let $\{\K_i\}_{i\in I}$ be a compact and stable 1-dimensional
filtration of $K$.Then the following statements hold:
\begin{enumerate}
\item $\alpha(P)\le\beta(P)$ for every $P\in K$.
\item If $P,Q\in K$ and $\alpha(P)<\alpha(Q)$,
then $\beta(P)\le\alpha(Q)$.
\item If $P,Q\in K$ and $\beta(P)<\beta(Q)$,
then $\beta(P)\le\alpha(Q)$.
\end{enumerate}
\end{lem}

\begin{proof}
\

\begin{enumerate}

\item To show that $\alpha(P)\le \beta(P)$, let us verify that, if
$i_1\in A(P)$ (i.e. $P\in\overline{\K_{i_1}^c}$) and $i_2\in B(P)$
(i.e. $P\in\K_{i_2}$), then $i_1\leq i_2$. By contradiction, let
us assume that $i_2<i_1$. Then Definition~\ref{stableDef}~$(b)$
implies that $\K_{i_2}\subseteq int(\K_{i_1})$. Since $P\in
\K_{i_2}$, it follows that $P\in int(\K_{i_1})$, i.e. $P\notin
\overline{\K_{i_1}^c}$, against the assumption $i_1\in A(P)$.
\item Since $\alpha(P)<\alpha(Q)$, it follows that $P\in int(\K_{\alpha(Q)})$, while $P\notin
int\left(\K_{\alpha(P)}\right)$.
In particular, $P\in \K_{\alpha(Q)}$. Therefore $\alpha(Q)\in
B(P)$ and hence $\beta(P)\le \alpha(Q)$. \item Since
$\beta(P)<\beta(Q)$, it follows that
$Q\notin \K_{\beta(P)}$, while $Q\in \K_{\beta(Q)}$. In particular, $Q\notin
int\left(\K_{\beta(P)}\right)$. Therefore $\beta(P)\in A(Q)$ and
hence $\beta(P)\le \alpha(Q)$.

\end{enumerate}
\end{proof}

\begin{rem}\label{boundary}
Let us observe that under the assumptions of compactness and
stability of $\{\K_i\}_{i\in I}$, it follows that, for every $P\in
K$ with $P\in\partial\K_i$ for a certain $i\in I$,
$\alpha(P)=\beta(P)=i$. Indeed, from Lemma
\ref{lemmaalphabeta}~(1), we have $\alpha(P)\le\beta(P)$ for every
$P\in K$. On the other side, since $P\in\partial\K_i$ implies both
that $P\in\K_i$, whence $\beta(P)\le i$, and that $P\notin
int(\K_i)$, whence $\alpha(P)\ge i$, the equality is proved.
\end{rem}

\begin{lem}\label{alphabeta}
Let $\{\K_i\}_{i\in I}$ be a compact and stable 1-dimensional
filtration of $K$.Then the following statements hold:
\begin{enumerate}
\item The function $\alpha$ is everywhere upper semi-continuous.
\item
The function $\beta$ is everywhere lower semi-continuous.
\end{enumerate}
\end{lem}

\begin{proof}  Let us consider a sequence $(P_r)$ of points in $K$ converging to
a point $P\in K$.
\begin{enumerate}
\item Let $(\alpha(P_{r_k}))$ be a converging subsequence of
$(\alpha(P_r))$. Let us set $L\overset{def}{=}\lim_k
\alpha(P_{r_k})$. From the compactness of $I$, $L\in I$, and from
Definition \ref{stableDef}~$(a)$, the sequence
$(\overline{\K_{\alpha(P_{r_k})}^c})$ converges to the compact set
$\overline{\K_{L}^c}$ with respect to $d_H$. Since $P=\lim_k
P_{r_k}$, and $P_{r_k}\in \overline{\K_{\alpha(P_{r_k})}^c}$, we
have that $P\in \overline{\K_{L}^c}$, and hence $\alpha(P)\ge L$.
Therefore, the function $\alpha$ is everywhere upper
semi-continuous.

\item Let $(\beta(P_{r_k}))$ be a converging subsequence of
$(\beta(P_r))$. Let us set $L\overset{def}{=}\lim_k
\beta(P_{r_k})$. From the compactness of $I$, $L\in I$, and from
Definition \ref{stableDef}~$(a)$, the sequence
$(\K_{\beta(P_{r_k})})$ converges to the compact set $\K_{L}$ with
respect to $d_H$. Since $P=\lim_k P_{r_k}$, and $P_{r_k}\in
\K_{\beta(P_{r_k})}$, we have that $P\in \K_{L}$, and hence
$\beta(P)\le L$. Therefore, the function $\beta$ is everywhere
lower semi-continuous.

\end{enumerate}
\end{proof}

\begin{thm}\label{monofil}
Every compact and stable $1$-dimensional filtration
$\{\K_i\}_{i\in I}$ of a compact metric space $K$ is induced by a
continuous function $\p:K\to\R$.
\end{thm}

\begin{proof}
If $\{\K_i\}_{i\in I}=\{\K_{i_{\min}}=\emptyset,
\K_{i_{\max}}=K\}$, then we can just take $\p:K\to\R$ such that
$\p(P)=i_{\max}$ for every $P\in K$. This function is continuous
and induces $\{\K_i\}_{i\in I}$.

Let us consider a proper filtration, i.e. a filtration
$\{\K_i\}_{i\in I}$ such that at least one index $i'\in I$ exists
with $i_{\min}<i'<i_{\max}$. We want to prove that there exists a
continuous function inducing it.

Let us observe that $\K_{i_{\min}}=\emptyset$ and, because of the
compactness of $I$, the value $i_1=\inf(I\setminus
\{i_{\min}\})\le i'$ must belong to $I$. The empty set cannot be
the limit of a sequence of compact non-empty sets with respect to
the Hausdorff distance. Hence it must be $i_1>i_{\min}$.
Furthermore, $\overline{\K_{i_{\max}}^c}=\overline{K^c}=\emptyset$
and, because of the compactness of $I$, the value
$i_2=\sup(I\setminus \{i_{\max}\})\geq i'$ must belong to $I$. The
empty set cannot be the limit of a sequence of compact non-empty
sets with respect to the Hausdorff distance. Hence it must be
$\K_{i_2}\neq\K_{i_{\max}}$, so that $i_2<i_{\max}$.

Now, let us fix an arbitrary point $*\notin K$, and extend the
distance $d$ from $K$ to $K\cup \{*\}$ by setting $d(*,*)=0$ and
$d(*,P)=\mbox{diam}(K)/2$ for every $P\in K$. Let us observe that
since the filtration is proper, $\mbox{diam}(K)>0$. Moreover, for
the same reason, we have that no point $P\in K$ exists such that
$\alpha(P)=i_{\min}$ and $\beta(P)=i_{\max}$. Hence, for every
$P\in K$, we can define the function $\p:K\to \R$ as follows, by
recalling the inequality in Lemma \ref{lemmaalphabeta}~$(1)$:
%
%
%

$$\p(P)=\left\{
\begin{array}{ll}
\beta(P) &\mbox{if}\,\, i_{\min}=\alpha(P)\\
\dfrac{\alpha(P)\cdot
d\left(P,\overline{\K_{\beta(P)}^{c}}\right)+\beta(P)\cdot
d\left(P,\K_{\alpha(P)}\right)}{d\left(P,\overline{\K_{\beta(P)}^{c}}\right)+d\left(P,\K_{\alpha(P)}\right)} & \mbox{if} \,\,  \small{i_{\min}<\alpha(P)<\beta(P)<i_{\max}}\\
\alpha(P) & \mbox{if} \,\,  \small{i_{\min}<\alpha(P)=\beta(P)<i_{\max}}\\
\dfrac{\alpha(P)\cdot d\left(P,*\right)+\beta(P)\cdot
d\left(P,\K_{\alpha(P)}\right)}{d\left(P,*\right)+d\left(P,\K_{\alpha(P)}\right)}
& \mbox{if} \,\,  \beta(P)=i_{\max}
\end{array}
\right.
$$

Before proceeding, we observe that
$d\left(P,\K_{\alpha(P)}\right)=0$ if and only if
$\alpha(P)=\beta(P)$, and also
$d\left(P,\overline{\K_{\beta(P)}^{c}}\right)=0$ if and only if
$\alpha(P)=\beta(P)$. Moreover, $\alpha(P)\le\p(P)\le\beta(P)$ in
all of the four cases in the definition of $\p$.

Let us prove that $\K_i=\{P\in K, \p(P)\leq i\}$ for every $i\in
I$.

Let us fix an index $i\in I$. If $P\in \K_i$ then $\beta(P)\le i$.
Hence, according to the observation above, $\p(P)\le \beta(P)\le
i$. Varying $i\in I$, this proves that $\K_i\subseteq\{P\in K,
\p(P)\leq i\}$ for every $i\in I$.

Let us show that $\K_i\supseteq\{P\in K, \p(P)\leq i\}$ for every
$i\in I$. If $P\notin \K_i$ then $P\in\K_i^c$, and hence $P\in
\overline{\K_i^c}$, so that $i\le \alpha(P)$. Since $P\notin
\K_i$, it follows that $\beta(P)>i$. Then, in all of the four
cases in the definition of $\p$ it is easy to show that $\p(P)>i$.
Therefore, in any case it results that $\p(P)> i$.



Now, let us show that $\p$ is continuous at any point $P\in K$.

First of all, let us examine the case $\alpha(P)=i_{\min}$ and the
case $\beta(P)=i_{\max}$.

If $\alpha(P)=i_{\min}$ then (since $i_1>i_{\min}$)
$\beta(P)=i_1$, and $P\in int(\K_{i_{1}})$ because of Remark
\ref{boundary}. So, there exists a neighborhood $U$ of $P$ such
that $U\subseteq int(\K_{i_{1}})$. It follows that for any point
$Q\in U$ the equalities $\alpha(Q)=i_{\min}$ and $\beta(Q)=i_{1}$
hold.

If $\beta(P)=i_{\max}$ then (since $i_2<i_{\max}$)
$\alpha(P)=i_2$, and $P\in int(\K_{i_{2}}^c)$ because of Remark
\ref{boundary}. So, there exists a neighborhood $U$ of $P$ such
that $U\subseteq int(\K_{i_{2}}^c)$. It follows that for any point
$Q\in U$ the equalities $\alpha(Q)=i_{2}$ and $\beta(Q)=i_{\max}$
hold.

In both cases, $\p$ is continuous at $P$.



\medskip


In the rest of the proof, we shall assume that
$i_{\min}<\alpha(P)$ and $\beta(P)<i_{\max}$.

In order to prove that $\p$ is continuous at $P$, it will be
sufficient to show that, if a sequence $(P_r)$ converges to $P$
and the sequence $(\varphi(P_r))$ is converging, then  $\lim_r
\p(P_r)=\p(P)$. This is due to the boundness of $\p(K)$.

Therefore, in what follows we shall assume that the sequences
$(P_r)$ and $(\varphi(P_r))$ are converging.

We recall that every real sequence admits either a strictly
monotone or a constant subsequence. Hence, by possibly extracting
a subsequence from $(P_r)$ we can assume that each of the
sequences $(\alpha(P_r)),(\beta(P_r))$ is either strictly monotone
or constant. Obviously, this choice does not change the limits of
the sequences $(P_r)$ and $(\varphi(P_r))$. Let us consider the
following two cases:

\medskip

\noindent\textbf{Case  that $(\beta(P_r))$ is strictly monotone:}
If $(\beta(P_r))$ is strictly decreasing, then
Lemma~\ref{lemmaalphabeta}~$(3)$ assures that
$\beta(P_{r+1})\le\alpha(P_r)$. As a consequence,
$$\varphi(P_{r+1})\le\beta(P_{r+1})\le\alpha(P_r)\le\varphi(P_r).$$

If $(\beta(P_r))$ is strictly increasing, then
Lemma~\ref{lemmaalphabeta}~$(3)$ assures that
$\beta(P_r)\le\alpha(P_{r+1})$. As a consequence,
$$\varphi(P_r)\le\beta(P_r)\le\alpha(P_{r+1})\le\varphi(P_{r+1}).$$

In both cases,
since the sequence $(\varphi(P_r))$ is converging,
also the sequences $(\alpha(P_r))$,
$(\beta(P_r))$ are converging and
$\lim_r\alpha(P_r)=\lim_r\beta(P_r)=\lim_r\varphi(P_r)$. Let us
call $\ell$ this limit.

The upper semi-continuity of the function $\alpha$ and the lower
semi-continuity of the function $\beta$ (Lemma \ref{alphabeta})
imply that $\alpha(P)\ge \ell\ge \beta(P)$. We already know that
$\alpha(P)\le\p(P)\le\beta(P)$, and hence
$\alpha(P)=\p(P)=\beta(P)=\ell$. Therefore,
$\varphi(P)=\lim_r\varphi(P_r)$.

\medskip

\noindent\textbf{Case  that $\beta(P_r)=L$ for every index $r$:}
If each element in the sequence $(\beta(P_r))$ is equal to a
constant $L$ then we know that, from the lower semi-continuity of
$\beta$ (Lemma \ref{alphabeta} $(2)$), $\beta(P)\le L\le i_{max}$.
\smallskip
\begin{itemize}
\item If $\beta(P)<L$, then there is no $h\in I$ such that
$\beta(P)<h<L$. Indeed, if such an index $h$ existed, Definition
\ref{stableDef}~$(b)$ would imply that $P\in\K_{\beta(P)}\subseteq
int(\K_h)$. Since $P=\lim_r P_r$, we would have that $P_r\in \K_h$
for every large enough index $r$. As a consequence, the inequality
$\beta(P_r)\le h<L$ would hold, against the assumption
$\beta(P_r)=L$ for every index $r$.

Lemma \ref{lemmaalphabeta} $(1)$ assures that $\alpha(P_r)\le
\beta(P_r)=L$ for every index $r$. Then, since $(\alpha(P_r))$ is
strictly monotone or constant, either $\alpha(P_r)=L$ for every
index $r$ or $\alpha(P_r)\le\beta(P)$ for every index $r>0$. We
observe that the case $\alpha(P_r)<\beta(P)$ cannot happen.
Indeed, if the inequality $\alpha(P_r)<\beta(P)$ held, then the
definition of $\alpha$ would imply that $P_r\in
int\left(\K_{\beta(P)}\right)\subseteq \K_{\beta(P)}$. As a
consequence, the inequality $L=\beta(P_r)\le \beta(P)$ would hold,
against the assumption $\beta(P)<L$.

In summary, if $\beta(P)<L$, then either $\alpha(P_r)=L$ for every
index $r$ or $\alpha(P_r)=\beta(P)$ for every index $r>0$, so that
$(\alpha(P_r))_{r>0}$, and therefore $(\alpha(P_r))$, is a
constant sequence.

Let us consider the following two subcases:
\begin{itemize}
\item Subcase $\alpha(P_r)=\beta(P_r)=L>\beta(P)$ for every $r$:
In this case, the upper semi-continuity of $\alpha$ implies that
$\alpha(P)\ge \lim_r\alpha(P_r)=L$, and hence that
$\alpha(P)>\beta(P)$, contradicting Lemma
\ref{lemmaalphabeta}~$(1)$. So this case is impossible.

\item Subcase $\alpha(P_r)=\beta(P)<\beta(P_r)=L$ for every $r$:
In this case, the upper semi-continuity of $\alpha$ implies that
$\alpha(P)\ge \lim_r\alpha(P_r)=\beta(P)$. Since Lemma
\ref{lemmaalphabeta} $(1)$ states that $\alpha(P)\le \beta(P)$, we
have $\alpha(P)=\beta(P)$. In summary, in this case,
$\alpha(P_r)=\alpha(P)=\beta(P)<\beta(P_r)=L$ for every index $r$.
From the definition of the function $\p$, it follows that
$\varphi(P)=\alpha(P)=\beta(P)$. Let us observe that
$\alpha(P_r)>i_{\min}$, otherwise $\alpha(P)=\beta(P)=i_{\min}$,
i.e. $P\in \K_{\beta(P)}=\K_{i_{\min}}$ in contrast with
$\K_{i_{\min}}=\emptyset$. Moreover, since $\beta(P_r)=L\le
i_{\max}$ for every index $r$, the two cases below must be
considered:

If $L<i_{\max}$, then
{\setlength\arraycolsep{1pt}\begin{eqnarray*}
\p(P_r)&=&\dfrac{\alpha(P_r)\cdot
d\left(P_r,\overline{\K_{\beta(P_r)}^{c}}\right)+\beta(P_r)\cdot
d\left(P_r,\K_{\alpha(P_r)}\right)}{d\left(P_r,\overline{\K_{\beta(P_r)}^{c}}\right)+d\left(P_r,\K_{\alpha(P_r)}\right)}\\
&=&\dfrac{\beta(P)\cdot
d\left(P_r,\overline{\K_{L}^{c}}\right)+L\cdot
d\left(P_r,\K_{\beta(P)}\right)}{d\left(P_r,\overline{\K_{L}^{c}}\right)+d\left(P_r,\K_{\beta(P)}\right)}.
 \end{eqnarray*}}

If $L=i_{\max}$, then
{\setlength\arraycolsep{1pt}\begin{eqnarray*}
\p(P_r)&=&\dfrac{\alpha(P_r)\cdot
d\left(P_r,*\right)+\beta(P_r)\cdot
d\left(P_r,\K_{\alpha(P_r)}\right)}{d\left(P_r,*\right)+d\left(P_r,\K_{\alpha(P_r)}\right)}\\
&=&\dfrac{\beta(P)\cdot d\left(P_r,*\right)+i_{\max}\cdot
d\left(P_r,\K_{\beta(P)}\right)}{d\left(P_r,*\right)+d\left(P_r,\K_{\beta(P)}\right)},
 \end{eqnarray*}}
with $*$ an arbitrary point not belonging to $K$, and such that
$d(*,Q)=\mbox{diam}(K)/2$ for every $Q\in K$.

Since $P\in \K_{\beta(P)}$ and $\lim_r P_r=P$, we have $\lim_r
d(P_r,\K_{\beta(P)})=0$. Furthermore, if $L<i_{\max}$, then
$\overline{\K_{L}^{c}}\neq\emptyset$, and $\lim_r
d(P_r,\overline{\K_{L}^{c}})=d(P,\overline{\K_{L}^{c}})>0$; if
$L=i_{\max}$, let us observe that $d(P_r,*)=d(P,*)>0$ for every
$P_r\in K$. Therefore, in both cases,
$\lim_r\p(P_r)=\beta(P)=\p(P)$, i.e. $\p$ is continuous at $P$.

\end{itemize}

\medskip

\item If $\beta(P)=L$, then $L< i_{max}$ (since we are assuming $\beta(P)<i_{\max}$). Recalling that
$(\alpha(P_r))$ is either a strictly monotone or a constant
bounded sequence, let $L'=\lim_r\alpha(P_r)$.

If the sequence $(\alpha(P_r))$ were strictly monotone, we could
find two indexes $r_1,r_2$ such that
$\alpha(P_{r_1}),\alpha(P_{r_2})\neq L$ and
$\alpha(P_{r_1})<\alpha(P_{r_2})$. Lemma \ref{lemmaalphabeta}
assures that $\beta(P_{r_1})\le\alpha(P_{r_2})\le\beta(P_{r_2})$.
Since $\beta(P_{r_1})=\beta(P_{r_2})=L$, it follows that
$\alpha(P_{r_2})=L$, against our assumption that
$\alpha(P_{r_1}),\alpha(P_{r_2})\neq L$. Therefore, the sequence
$(\alpha(P_r))$ must be constant.

In summary, if $\beta(P_r)=\beta(P)=L$ for every index $r$, then $\alpha(P_r)=L'$ for every index $r$.

Since the function $\alpha$ is upper semi-continuous (Lemma
\ref{alphabeta} $(1)$), we have that $\alpha(P)\ge L'$. If the
inequality $\alpha(P)> L'$ holds, then $\alpha(P_r)<\alpha(P)$ for
every index $r$. Lemma \ref{lemmaalphabeta} $(2)$ assures that
$\beta(P_r)\le \alpha(P)$, and hence $\alpha(P)\ge L$. Lemma
\ref{lemmaalphabeta} $(1)$ assures that $\alpha(P)\le\beta(P)$,
and hence $\alpha(P)\le L$. Therefore, $\alpha(P)=L$.

In summary, if $\beta(P_r)=\beta(P)=L$ for every index $r$, then either $\alpha(P)=L'$ or $\alpha(P)=L$.
\medskip

Therefore, we have to examine these last three cases:

\medskip

\begin{itemize}
\item[$(i):$] $\beta(P_r)=\beta(P)=L>\alpha(P_r)=\alpha(P)=L'$ for
every index $r$; \item[$(ii):$]
$\beta(P_r)=\beta(P)=\alpha(P)=L>\alpha(P_r)=L'$ for every index
$r$; \item[$(iii):$]
$\beta(P_r)=\beta(P)=\alpha(P)=L=\alpha(P_r)=L'$ for every index
$r$.
\end{itemize}
\medskip

\begin{itemize}
\item[$(i):$] If $\beta(P_r)=\beta(P)=L>\alpha(P_r)=\alpha(P)=L'$
for every $r$, recalling that
$i_{\min}<\alpha(P)<\beta(P)<i_{\max},$ the definition of the
function $\p$ implies that

{\setlength\arraycolsep{1pt}\begin{eqnarray*}
\p(P_r)&=&\dfrac{\alpha(P_r)\cdot
d(P_r,\overline{\K_{\beta(P_r)}^{c}})+\beta(P_r)\cdot
d(P_r,\K_{\alpha(P_r)})}{d(P_r,\overline{\K_{\beta(P_r)}^{c}})+d(P_r,\K_{\alpha(P_r)})} \\
&=&\dfrac{L'\cdot d(P_r,\overline{\K_{L}^{c}})+L\cdot
d(P_r,\K_{L'})}{d(P_r,\overline{\K_{L}^{c}})+d(P_r,\K_{L'})}
 \end{eqnarray*}}
while
\begin{eqnarray*}
\p(P)=\dfrac{L'\cdot d(P,\overline{\K_{L}^{c}})+L\cdot
d(P,\K_{L'})}{d(P,\overline{\K_{L}^{c}})+d(P,\K_{L'})}.
\end{eqnarray*}
 Therefore $\lim_r\p(P_r)=\p(P)$, and
 hence the function $\p$ is continuous at $P$.

\item[$(ii):$] If $\beta(P_r)=\beta(P)=\alpha(P)=L>\alpha(P_r)=L'$
for every index $r$, the definition of the function $\p$ implies
that, in the case $\alpha(P_r)>i_{\min}$,
{\setlength\arraycolsep{1pt}\begin{eqnarray*}
\p(P_r)&=&\dfrac{\alpha(P_r)\cdot
d(P_r,\overline{\K_{\beta(P_r)}^{c}})+\beta(P_r)\cdot
d(P_r,\K_{\alpha(P_r)})}{d(P_r,\overline{\K_{\beta(P_r)}^{c}})+d(P_r,\K_{\alpha(P_r)})} \\
&=&\dfrac{L'\cdot d(P_r,\overline{\K_{L}^{c}})+L\cdot
d(P_r,\K_{L'})}{d(P_r,\overline{\K_{L}^{c}})+d(P_r,\K_{L'})},
 \end{eqnarray*}}
otherwise, if $\alpha(P_r)=i_{\min}$, $\p(P_r)=\beta(P_r)=L$.
Recalling that $P\in
\overline{\K_{\alpha(P)}^c}=\overline{\K_{L}^c}$ and $\lim_r
P_r=P$, it follows that, in both cases, $\lim_r\p(P_r)=L$. On the
other hand $\p(P)=\alpha(P)=L.$
 Therefore $\lim_r\p(P_r)=\p(P)$, and
 hence the function $\p$ is continuous at $P$.
\item[$(iii):$] If
$\beta(P_r)=\beta(P)=\alpha(P)=L=\alpha(P_r)=L'$ for every index
$r$, the definition of the function $\p$ implies that
$\p(P_r)=\p(P)=L$ for every index $r$. Therefore
$\lim_r\p(P_r)=\p(P)$, and
 hence the function $\p$ is continuous at $P$ also in this case.
\end{itemize}
\end{itemize}

\end{proof}

Let us observe that, dropping the assumption of stability
(Definition~\ref{stableDef}), Theorem~\ref{monofil} does not hold,
as the following examples show. The first one does not verify
property $(a)$ in Definition \ref{stableDef}, the second one does
not verify property $(b)$ in Definition \ref{stableDef}.
\begin{exa}\label{exastab-a}
Let $K$ be the closed interval $[0,2]$, and $I=\{-1\}\cup[0,1]$. Let us
consider the compact sets
$$
\K_i=\left\{\begin{array}{ll}
\emptyset &\mbox{if}\,\,i=-1\\
\{0\}&\mbox{if}\,\,i=0\\
{[0,i+1]} &\mbox{if}\,\,i \in ]0,1].
\end{array}
\right.
$$
This filtration of $K$ is not stable because, contrarily to
Definition \ref{stableDef}~$(a)$, when the index $i$ tends to 0,
the compact sets $\K_i$ do not tend to $\K_0$.

Let us show that this filtration of the interval $K$ cannot be
induced by any function $\p:K\to\R$. Indeed, if such a function
$\p$ existed, we would have $\p(P)\le \eps$ for every $\eps>0$ and
every $P\in[0,1]$ since $[0,1]\subseteq \K_{\eps}$ for every
$\eps>0$. Therefore, $\p$ would take a non-positive value at each
$P\in[0,1]$, against the equality $\K_0=\{0\}$.
\end{exa}
\begin{exa}\label{exastab-b}
Let $K$ be the disk filtered by the family $\{\K_0, \K_1,\K_2,
\K_3\}$ in Figure~\ref{stability}, with $\K_0=\emptyset$ and
$\K_3=K$. This filtration of $K$ is not stable because, contrarily
to Definition~\ref{stableDef}~$(b)$, $\K_1\nsubseteq int(\K_2)$.
\begin{figure}[htbp]
\begin{center}
\psfrag{K}{$\K_3\equiv
K$}\psfrag{K1}{$\K_{1}$}\psfrag{K2}{$\K_{2}$} \psfrag{P}{$\bar P$}
\includegraphics[height=5cm]{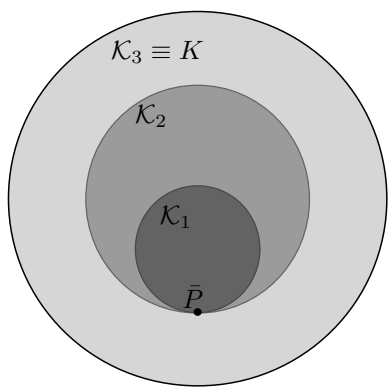}
\caption{\footnotesize{An example of non-stable $1$-dimensional
filtration of the disk $K$.}\label{stability}}
\end{center}
\end{figure}
Let us show that this filtration of the disk $K$ cannot be induced
by a continuous function $\p:K\to\R$. Indeed, if such a continuous
function $\p$ existed, it should be that $\p(\bar P)\le 1$, since
$\bar P\in \K_1$.  On the other hand, if we consider a sequence
$(P_r)$ of points of $\K_3\setminus \K_2$ that converges to $\bar
P$, we should have $\p(\bar P)=\lim_r \p(P_r)\ge 2$ (since
$\p(P_r)>2$ for every index $r$, given that $P_r\notin K_2$). This
contradiction proves our statement.
\end{exa}

\section{Multi-dimensional filtrations}\label{Multi}
In this section, we extend the main result of Section \ref{Mono}
to $n$-dimensional filtrations, $n\geq 1$, i.e. to the case of
filtrations indexed by an $n$-dimensional parameter. Therefore, in
what follows, the symbol $I$ will denote a compact subset $I_1\times
I_2\times\ldots\times I_n$ of $\R^n$ and $p_j:I\to I_j$, $1\le
j\le n$, the projection of $I$ onto the $j$-th component.

For every fixed $j$ with $1\le j\le n$ and every $h\in I_j$, let
us set
$$
\K_h^j= \K_{(\max I_1,\ldots,\max I_{j-1},h,\max
I_{j+1},\ldots,\max I_n)}= \bigcup_{\substack{i\in I\\
p_j(i)=h}}\K_i.$$ We observe that $\{\K_{h}^j\}_{h\in I_j}$ is a
compact $1$-dimensional filtration of $K$.

\begin{defi}\label{defMulStability}
We shall say that a compact $n$-dimensional filtration
$\{\K_{i}\}_{i\in I}$ of $K$ is \emph{stable with respect to} $d$
if the compact $1$-dimensional filtrations $\{\K_{i_1}^1\}_{i_1\in
I_1}$, $\{\K_{i_2}^2\}_{i_2\in I_2}$, $\ldots$,
$\{\K_{i_n}^n\}_{i_n\in I_n}$ are stable with respect to $d$.
\end{defi}

\begin{defi}\label{defComplete}
A compact $n$-dimensional filtration $\{\K_{i}\}_{i\in I}$ of $K$
will be said to be \emph{complete} if, for every
$i=(i_1,\ldots,i_n)\in I$, $\K_{i}=
\K_{i_1}^1\cap\K_{i_2}^2\cap\ldots\cap\K_{i_n}^n$.
\end{defi}

\begin{rem}\label{remComplete}
Let us observe that, setting $i_{min}=(\min I_1,\min
I_2,\dots,\min I_n)$ and $i_{max}=(\max I_1,\max I_2,\dots,\max
I_n)$, Definition \ref{defComplete} implies that
$\K_{i_{min}}=\K_{\min I_1}^1\cap\K_{\min
I_2}^2\cap\ldots\cap\K_{\min
I_n}^n=\emptyset\cap\emptyset\cap\ldots\cap\emptyset=\emptyset$,
and $\K_{i_{max}}=\K_{\max I_1}^1\cap\K_{\max
I_2}^2\cap\ldots\cap\K_{\max I_n}^n=K\cap K\cap\ldots\cap K=K$.
\end{rem}

\begin{thm}\label{multfil}
Every compact, stable and complete $n$-dimensional filtration
$\{\K_{i}\}_{i\in I}$ of a compact metric space $K$ is induced by
a continuous function $\vp:K\to\R^n$.
\end{thm}

\begin{proof}
By Definition \ref{defComplete}, the completeness of
$\{\K_{i}\}_{i\in I}$ implies that, for every
$i=(i_1,i_2,\ldots,i_n)\in I$, $\K_{i}$ is equal to
$\K_{i_1}^1\cap\K_{i_2}^2\cap\ldots\cap\K_{i_n}^n$. Moreover, by
Definition~\ref{defMulStability}, the stability of
$\{\K_{i}\}_{i\in I}$ implies the stability of the $1$-dimensional
filtrations $\{\K_{i_1}^1\}_{i_1\in I_1}$, $\{\K_{i_2}^2\}_{i_2\in
I_2}$, $\ldots$, $\{\K_{i_n}^n\}_{i_n\in I_n}$. Then, by Theorem
\ref{monofil}, for every $\{\K_{i_j}^j\}_{i_j\in I_j}$,
$j=1,\ldots,n$, there exists a continuous function $\p_j:K\to\R$
such that $\K_{i_j}^j=\{P\in K:\p_j(P)\leq i_j\}$ for every
$i_j\in I_j$. Hence, {\setlength\arraycolsep{1pt}\begin{eqnarray*}
&&\K_{(i_1,i_2,\ldots,i_n)}=\K_{i_1}^1\cap\K_{i_2}^2\cap\ldots\cap\K_{i_n}^n\\
&=& \{P\in K:\p_1(P)\leq i_1\}\cap\{P\in K:\p_2(P)\leq
i_2\}\cap\ldots\cap\{P\in K:\p_n(P)\leq i_n\}\\
&=&\{P\in
K:\vp(P)=(\p_1,\p_2,\ldots,\p_n)(P)\preceq(i_1,i_2,\ldots,i_n)\}.
\end{eqnarray*}}
Therefore, the function $\vp:K\to\R^n$ induces
$\{\K_{i}\}_{i\in I}$. Moreover, $\vp$ is continuous since its components
$\p_1,\p_2,\ldots,\p_n:K\to\R$ are continuous.
\end{proof}

Let us observe that, without the assumption of completeness
(Definition \ref{defComplete}), Theorem \ref{multfil} does not
hold, as the following example shows.

\begin{exa}\label{exacomplete}
Let $K$ be the rectangle in Figure \ref{completeness}, filtered by
the family $\{\K_{(i_1,i_2)}\}$, with $(i_1,i_2)$ varying in the
set $I=\{0,1,2\}\times\{0,1,2\}$. From Remark \ref{remComplete},
we have $\K_{(0,i)}=\K_{(i,0)}=\emptyset$ for $i=0,1,2$, and
$\K_{(2,2)}=K$. We observe that $\{\K_{(i_1,i_2)}\}_{(i_1,i_2)\in
I}$ is stable since the $1$-dimensional filtrations
$\{\K_{i_1}^1\}_{i_1\in\{0,1,2\}}=\{\K_{(0,2)},\K_{(1,2)},\K_{(2,2)}\}$,
and
$\{\K_{i_2}^2\}_{i_2\in\{0,1,2\}}=\{\K_{(2,0)},\K_{(2,1)},\K_{(2,2)}\}$
are stable with respect to $d$. However,
$\{\K_{(i_1,i_2)}\}_{(i_1,i_2)\in I}$ is not complete since
$\K_{(1,1)}\subsetneqq\K_{(2,1)}\cap\K_{(1,2)}$.
\begin{figure}[htbp]
\begin{center}
\psfrag{K}{$\K_{(2,2)}\equiv K$}\psfrag{K11}{$\K_{(1,1)}$}
\psfrag{K12}{$\K_{(1,2)}$}\psfrag{K21}{$\K_{(2,1)}$}
\psfrag{P}{$\bar P$}
\includegraphics[height=5cm]{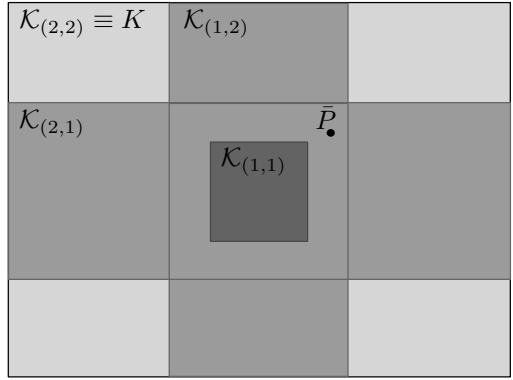}
\caption{\footnotesize{An example of non-complete 2-dimensional
filtration of the rectangle $K$.}\label{completeness}}
\end{center}
\end{figure}

Let us show that this 2-dimensional filtration of the rectangle
$K$ cannot be induced by a continuous function $\vp:K\to\R^2$.

Let $\bar P\in\K_{(1,2)}\cap \K_{(2,1)}\setminus\K_{(1,1)}$ as in
Figure \ref{completeness}. If there existed
$\vp=(\p_1,\p_2):K\to\R^2$ inducing this filtration, then
$\p_1(\bar P)\leq 1$ because $\bar P\in \K_{(1,2)}$, and
$\p_2(\bar P)\leq 1$ because $P\in \K_{(2,1)}$. Therefore,
$\vp(\bar P)\preceq (1,1)$. This means that $\bar P$ should belong
to $\K_{(1,1)}$, giving a contradiction.
\end{exa}

\section{Comparing filtrations via functions}\label{distance}
In the application of persistent homology to the problem of shape
comparison, it is natural to estimate the shape dissimilarity of
two spaces starting from the comparison of filtrations defined on
them. Nowadays, this problem is usually tackled by computing the
bottleneck distance between the persistence diagrams associated
with each filtration. Unfortunately, the loss of information due
to the passage from filtrations to persistence diagrams often
makes these descriptors unable to distinguish different shapes
(see e.g. \cite{CaFePo01,FrLa11,DiLa12}).

As proved in the previous sections, under appropriate assumptions,
every filtration of a compact space is induced by at least one
continuous function. Hence, by virtue of this fact, it is possible
to directly compare two filtrations by computing distances between
the associated filtering functions defined as in Theorem
\ref{monofil} (for the case $n=1$) and Theorem \ref{multfil} (for
the case $n>1$). For example, we could use the natural
pseudo-distance if we are interested in the functions' invariance
under the action of homeomorphisms, or the $L$-infinity distance
if this is not the case.

Let us recall that the natural pseudo-distance between two
continuous functions $\p,\p':K\to\R^n$ is defined as
$\delta(\p,\p')=\underset{h\in\mathcal{H}(K)}\inf\|\p-\p'\circ
h\|_{\infty}$, where $\mathcal{H}(K)$ denotes the set of all
self-homeomorphisms of $K$ \cite{DoFr04,DoFr07,DoFr09,FrMu99}. The
use of the natural pseudo-distance implies that the distance
between the filtrations induced by $\p$ and $\p\circ h$, with
$h\in\mathcal{H}(K)$, vanishes, so that these filtrations are
considered equivalent.
%
%
%
%
%
\begin{figure}[htbp]
\begin{center}
\psfrag{a}{$1$} \psfrag{b}{$2$}\psfrag{c}{$3$} \psfrag{B2}{$p_2$}
\psfrag{d}{$4$} \psfrag{e}{$5$} \psfrag{k}{$6$} \psfrag{f}{$\p$}
\psfrag{f'}{$\p'$} \psfrag{K1}{$\K_1$}
\psfrag{K2}{$\K_2$}\psfrag{K3}{$\K_3$}\psfrag{K4}{$\K_4$}\psfrag{K5}{$\K_5$}\psfrag{K6}{$\K_6$}\psfrag{K'1}{$\K'_1$}
\psfrag{K'2}{$\K'_2$}\psfrag{K'3}{$\K'_3$}\psfrag{K'4}{$\K'_4$}\psfrag{K'5}{$\K'_5$}\psfrag{K'6}{$\K'_6$}
\psfrag{f}{$\p$} \psfrag{g}{$\p'$} \psfrag{M}{$\max
z$}\psfrag{0}{$0$}\psfrag{2p}{$2\pi$}\psfrag{R}{$\scriptstyle{R}$}\psfrag{R'}{$R'$}
\psfrag{m}{$\min
z$}\psfrag{A}{$A$}\psfrag{B}{$\scriptstyle{B}$}\psfrag{A'}{$\scriptstyle{A'}$}\psfrag{B'}{$\scriptstyle{B'}$}\psfrag{C}{$\scriptstyle{C}$}
\psfrag{C'}{$\scriptstyle{C'}$}\psfrag{D}{$\scriptstyle{D}$}\psfrag{D'}{$\scriptstyle{D'}$}\psfrag{E}{$\scriptstyle{E}$}\psfrag{F}{$\scriptstyle{F}$}
\psfrag{G}{$\scriptstyle{G}$}\psfrag{H}{$\scriptstyle{H}$}
\psfrag{O}{$O$}\psfrag{T'}{$\scriptstyle{T'}$}\psfrag{Q}{$\scriptstyle{Q}$}\psfrag{Q'}{$\scriptstyle{Q'}$}\psfrag{S'}{$\scriptstyle{S'}$}
\includegraphics[width=12cm]{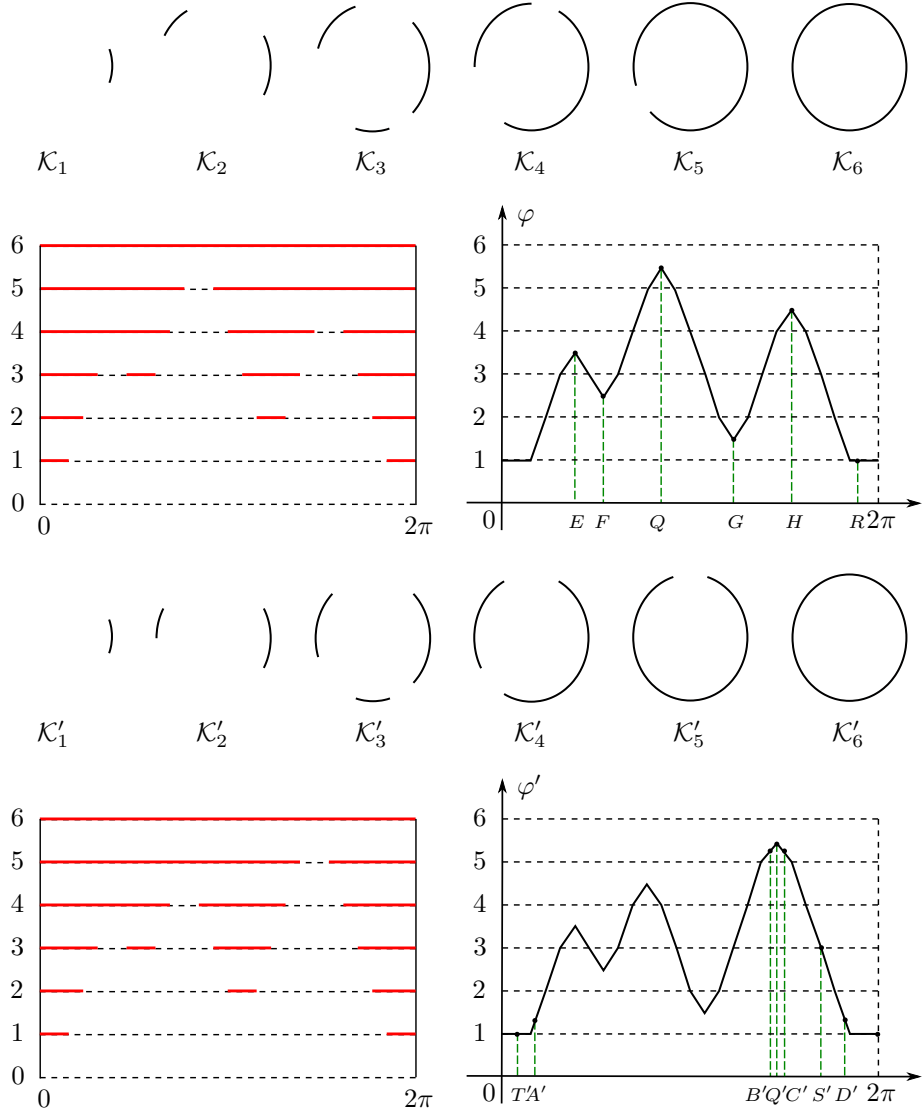}
\caption{\footnotesize{Two filtrations of $S^1$ with the same
persistence diagrams but a non-zero natural pseudo-distance
between the inducing functions.}}\label{esempio_tri}
\end{center}
\end{figure}

The choice of comparing two filtrations in terms of the natural
pseudo-distance between the associated filtering functions results
to be more powerful than the bottleneck distance between the
associated persistence diagrams, in distinguishing two different
filtrations. We will show this fact through an example inspired to
the one in \cite{DiLa12}.

Let $K$ be the circle $S^1$, and consider the two stable finite
filtrations $\{\K_i\}$ and $\{\K'_i\}$ shown in Figure
\ref{esempio_tri} which are defined on the set of indices
$I=\{0,1,2,3,4,5,6\}$, and are such that
$\K_0\equiv\K'_0\equiv\emptyset$, $\K_6\equiv\K'_6\equiv K$. Let
us construct $\p,\p':K\to\R$ as in the proof of
Theorem~\ref{monofil},
defining on $K$ the geodesic
distance $d$, and extending it to an arbitrary point $*\notin K$ by setting
$d(*,P)=\mbox{diam}(K)/2=\pi/2$ for every $P\in K$. The value of $\p$ and $\p'$ at each point of $K$
is its ordinate in the real plane. 

While the bottleneck distance between the persistence diagrams
associated with $\{\K_i\}$ and $\{\K'_i\}$ is zero in all homology
degrees, let us prove that $\delta(\p,\p')$ is positive. By
contradiction, let us assume that $\delta(\p,\p')=0$. Then, for
every $\eps>0$ sufficiently small, there should exist a
homeomorphism $h_{\eps}:K\to K$ such that $\underset{P\in K}\max
|\p(P)-\p'\circ h_{\eps}(P)|\leq \eps$. Such a homeomorphism
should take all the points of maximum (minimum, respectively) of
$\p$ to points near the points of maximum (minimum, respectively)
of $\p'$ with the same ordinate. Therefore, denoting by
$\bigfrown{xyz}$ the arc of $S^1$ which contains $y$ and has $x$
and $z$ as its endpoints, the points $Q$ and $R$ in Figure
\ref{esempio_tri} should be taken to
$h_{\eps}(Q)\in\bigfrown{B'Q'C'}$ and
$h_{\eps}(R)\in\bigfrown{D'T'A'}$, respectively, where
$A',B',C',D'$ are such that $\p'(B')=\p'(C')=\p(Q)-\eps$,
$\p'(A')=\p'(D')=\p(R)+\eps$. Hence, either
$\bigfrown{h_{\eps}(Q)S'h_{\eps}(R)}=h_{\eps}(\bigfrown{QGR})$ or
$\bigfrown{h_{\eps}(Q)S'h_{\eps}(R)}=h_{\eps}(\bigfrown{QER})$. As
proved in \cite{DiLa12}, in the first case $\underset{P\in
\scriptsize{\bigfrown{QGR}}}\max |\p(P)-\p'\circ h_{\eps}(P)|>
\dfrac{|\p(G)-\p(H)|}{2}$; in the second case,  $\underset{P\in
\scriptsize{\bigfrown{QER}}}\max |\p(P)-\p'\circ h_{\eps}(P)|>
\dfrac{|\p(E)-\p(F)|}{2}$. In conclusion, $\underset{P\in K}\max
|\p(P)-\p'\circ
h_{\eps}(P)|>\min\left\{\dfrac{|\p(G)-\p(H)|}{2},\dfrac{|\p(E)-\p(F)|}{2}\right\}>
\eps$, against the assumption. Hence persistent homology is not
able to distinguish the two considered shapes, contrarily to the
natural pseudo-distance between the associated filtering
functions.


\subsection*{Acknowledgments}
The authors gratefully acknowledge the anonymous referee for
her/his help in improving the paper. Thanks also to I. Halevy and
D. Burghelea for their helpful comments and stimulating
conversations. However, the authors are solely responsible for any
possible errors.

Finally, they wish to express their gratitude to M. Ferri, C.
Landi, and A. Cerri for their indispensable support and
friendship.

This paper is dedicated to Filippo and Sara.
\bibliographystyle{amsplain}
\bibliography{Filbiblio}

\end{document}